\documentclass[a4paper]{article}
\usepackage[affil-it]{authblk}
\usepackage{amsthm,amsmath,amssymb,amsfonts,mathrsfs}
\usepackage[square, numbers, comma]{natbib}
\usepackage{enumerate}
\usepackage{amsbsy}
\usepackage{amsfonts}
\usepackage[all,cmtip]{xy}
\usepackage{tikz}

\makeatletter
\renewcommand\paragraph{\@startsection{paragraph}{4}{\z@}%
            {-2.5ex\@plus -1ex \@minus -.25ex}%
            {1.25ex \@plus .25ex}%
            {\normalfont\normalsize\bfseries}}
            \makeatother
\usetikzlibrary{arrows,matrix}

\usepackage{pb-diagram,pb-xy}
   \bibpunct{[}{]}{,}{n}{}{,}
\input cyracc.def  

\def\Q{\mathbb{Q}}
\def\F{\mathbb{F}}

\def\Z{\mathbb{Z}}

\def\p{\mathfrak{p}}
\def\q{\mathfrak{q}}

\def\Aut{\mathrm{Aut}}

\def\Gal{\mathrm{Gal}}

\def\dim{\mathrm{dim}}

\def\Cl{\mathrm{Cl}}
\def\GL{\mathrm{GL}}
\def\PGL{\mathrm{PGL}}
\def\SL{\mathrm{SL}}
\def\PSL{\mathrm{PSL}}

\newcommand{\sm}{\left(\begin{smallmatrix}}
\newcommand{\esm}{\end{smallmatrix}\right)}

\newtheorem{thm}{Theorem}[section]
\newtheorem{lem}[thm]{Lemma}
\newtheorem{prop}[thm]{Proposition}
\newtheorem{cor}[thm]{Corollary}

\theoremstyle{definition}
\newtheorem{dfn}[thm]{Definition}

\theoremstyle{remark}

\numberwithin{equation}{section}

\begin{document}
\title{Some examples of quadratic fields with finite nonsolvable maximal unramified extensions II}

\author{Kwang-Seob Kim\thanks{Email: \texttt{kwang12@kias.re.kr}, School of Mathematics, Korea Institute for Advanced Study, Seoul 130-722, Korea} \ 
and
Joachim K\"onig\thanks{
Email: \texttt{koenig.joach@technion.ac.il}, Department of Mathematics, Technion Israel Institute of Technology, 32000 Haifa, Israel
\\
}
}




\maketitle

\begin{abstract}
Let $K$ be a number field and $K_{ur}$ be the maximal extension of $K$ that is unramified at all places. 
In the previous article \cite{Kim-2016}, the first author found three real quadratic fields $K$ such that $\Gal(K_{ur}/K)$ is finite and nonabelian simple under the assumption of the GRH(Generalized Riemann Hypothesis). In this article, we will identify more quadratic number fields $K$ such that $\Gal(K_{ur}/K)$ is a finite nonsolvable group and also explicitly calculate their Galois groups under the assumption of the Generalized Riemann Hypothesis.
\end{abstract}

\section{Introduction}
This is a continuation of \cite{Kim-2016}. Let $K$ be a number field and $K_{ur}$ be the maximal extension of $K$ that is unramified at \emph{all places}. In \cite{Yamamura-1997}, Yamamura showed that $K_{ur}=K_l$, where $K$ denotes an imaginary quadratic field with absolute  discriminant value $|d_K| \leq 420$, and $K_l$ is the top of the class field tower of $K$ and also computed $\Gal(K_{ur}/K)$. Hence, we can find examples of abelian or solvable \'etale fundamental groups. It is then natural to wonder whether we can find  examples with the property that $\Gal(K_{ur}/K)$ is a finite nonsolvable group. In the previous article \cite{Kim-2016}, we present three explicit examples that provide an affirmative answer.

In this article, we will identify two more quadratic number fields $K$ such that $\Gal(K_{ur}/K)$ is a finite nonsolvable group and also explicitly calculate their Galois groups under the GRH. Under the assumption of GRH, we will show that $\Gal(K_{ur}/K)$ is isomorphic to a finite nonsolvable group when $K=\Q(\sqrt{22268})$ (Theorem \ref{thm:22268}) and when $K=\Q(\sqrt{-1567})$ (Theorem \ref{thm:1567}).

In particular, to the best of the authors' knowledge, $K=\Q(\sqrt{-1567})$ is the first example of an imaginary quadratic field which has a nonsolvable unramified extension and for which $\Gal(K_{ur}/K)$ is explicitly calculated.\\
\emph{Tools used for the proof}: To identify certain unramified extensions with nonsolvable Galois groups, we use the database of number fields created by J\"urgen Kl\"uners and Gunter Malle \cite{Kluners}.
To exclude further unramified extensions, we use a wide variety of tools, including class field theory, Odlyzko's discriminant bounds, results about low degree number fields with small discriminants, and various group-theoretical results. In particular, the group-theoretical arguments are far more involved than in the previous paper \cite{Kim-2016}.





\section{Preliminaries}\label{Preliminaries}
\subsection{The action of Galois groups on class groups}
If $A$ is a finite abelian $p$-group, then $A \simeq \oplus \Z/p^{a_i}\Z$ for some integers $a_i$. Let
\begin{displaymath}
n_a=\textrm{ number of }i\textrm{ with }a_i=a,
\end{displaymath}
\begin{displaymath}
r_a=\textrm{ number of }i\textrm{ with }a_i\geq a.
\end{displaymath}
Then
\begin{displaymath}
r_1=p\textrm{-rank }A\textrm{ }=\dim_{\Z/p\Z}(A/A^p)
\end{displaymath}
and, more generally,
\begin{displaymath}
r_a=\dim_{\Z/p\Z}(A^{p^{a-1}}/A^{p^a}).
\end{displaymath}
The action of Galois groups on class groups can often be used to obtain useful information on the structure of lass groups. We review the following lemma, often called $p$-rank theorem.
\begin{lem} \label{2.5}
(Theorem 10.8 of \cite{Washington-1982}) Let $L/K$ be a cyclic of degree $n$. Let $p$ be a prime, $p \nmid n$ and assume that all fields $E$ with $K \subseteqq E \varsubsetneqq L$ satisfy $p \nmid \mathrm{Cl}(E)$. Let $A$ be the $p$-Sylow subgroup of the ideal class group of $L$, and let $f$ be the order of $p$ mod $n$. Then
\begin{displaymath}
r_a \equiv n_a \equiv 0 \textrm{ mod }f
\end{displaymath}
for all $a$, where $r_a$ and $n_a$ are as above. In particular, if $p | \mathrm{Cl}(L)$ then the $p$-rank of $A$ is at least $f$ and $p^f | \mathrm{Cl}(L)$.
\end{lem}
\subsection{A remark on the class field tower}
\begin{lem} \label{2.6}
(Theorem 1 of \cite{Taussky-1937}) Let $K$ be an algebraic number field of finite degree and $p$ any prime number. If the $p$-class group, i.e., the $p$-part of the class group of $K$ is cyclic, then the $p$-class group of the Hilbert $p$-class field of $K$ is trivial. Moreover, if $p=2$ and the $2$-class group of $K$ is isomorphic to $V_4$, then the $2$-class group of the Hilbert $2$-class field of $K$ is cyclic.
\end{lem}
\subsection{Root discriminant}
Let $K$ be a number field. We define the \textit{root discriminant} of $K$ to be $|d_K|^{1/n_K}$, where $n_K$ is $[K:\Q]$. Given a tower of number fields $L/K/F$, we have the following equality for the ideals of $F$:
\begin{equation}
\begin{split} \label{2-1}
d_{L/F} = (d_{K/F})^{[L:K]} N_{K/F}(d_{L/K}),
\end{split}
\end{equation}
where $d_{L/F}$ denotes the relative discriminant (see Corollary 2.10 of \cite{Neukirch-1999}). Set $F=\Q$. It follows from (\ref{2-1}) that, if $L$ is an extension of $K$,  $|d_K|^{1/n_K} \leq |d_L|^{1/n_L}$, with equality if and only if $d_{L/K}=1$, i.e., $L/K$ is unramified at \emph{all finite places}.
\subsection{Discriminant bounds} \label{Discriminant bounds}
In this section, we describe how the discriminant bound is used to determine that a field has no nonsolvable unramified extensions.
\subsubsection{Crucial proposition}

Consider the following proposition, in which $K_{ur}$ is the maximal extension of $K$ that is unramified over all primes.
\begin{prop} \label{prop:y}
(Proposition 1 of \cite{Yamamura-1997}) Let $B(n_K, r_1, r_2)$ be the lower bound for the root discriminant of $K$ of degree $n_K$ with signature $(r_1, r_2)$.  Suppose that $K$ has an unramified normal extension $L$ of degree $m$. If $\mathrm{Cl}(L) = 1$, where $\mathrm{Cl}(L)$ is the class number of $L$, and $|d_K|^{1/n_K} < B(60mn_K, 60mr_1, 60mr_2)$, then $K_{ur} = L$.
\end{prop}
If the GRH is assumed, much better bounds can be obtained. 
The lower bounds for number fields are stated in Martinet's expository paper \cite{Martinet-1980}.
\subsubsection{Description of Table III of \cite{Martinet-1980}}
Table III of \cite{Martinet-1980} describes the following. If $K$ is an algebraic number field with $r_1$ real and $2r_2$ complex conjugate fields, and $d_K$ denotes the absolute value of the discriminant of $K$, then, for any $b$, we have
\begin{equation}
\begin{split} \label{D}
d_K>A^{r_1}B^{2r_2}e^{f-E},
\end{split}
\end{equation}
where $A,B$, and $E$ are given in the table, and
\begin{equation}
\begin{split} \label{f}
f=2\sum_{\p}\sum^{\infty}_{m=1}\frac{\log N(\p)}{N(\p)^{m/2}}F(\log N(\p)^m),
\end{split}
\end{equation}
where the outer sum is taken over all prime ideals of $K$, $N$ is the norm from $K$ to $\Q$, and
\begin{displaymath}
F(x)=G(x/b)
\end{displaymath}
in the GRH case, where the even function $G(x)$ is given by
\begin{equation}
\begin{split}
G(x) = \Big(1- \frac{x}{2}\Big)\cos\frac{\pi}{2}x+\frac{1}{\pi}\sin\frac{\pi}{2}x
\end{split}
\end{equation}
for $0 \leq x \leq 2$ and $G(x) = 0$ for $x>2$.

The values of $A$ and $B$ are lower estimates; the values of $E$ have been rounded up from their true values, which are
\begin{equation}
\begin{split}
8\pi^2b\Big(\frac{e^{b/2}+e^{-b/2}}{\pi^2+b^2}\Big)^2
\end{split}
\end{equation}
in the GRH case.
\section{Some group theory} \label{group}
In this section, we recall some facts from  group theory.
\subsection{Schur multipliers and central extensions}
\begin{dfn}
The \emph{Schur multiplier} is the second homology group $H_2(G, \Z)$ of a group $G$.
\end{dfn}
\begin{dfn}
A \emph{stem extension} of a group $G$ is an extension
\begin{equation}
\begin{split}
1 \to H \to G_0 \to G \to 1,
\end{split}
\end{equation}
where $H \subset Z(G_0) \cap G'_0$ is a subgroup of the intersection of the center of $G_0$ and the derived subgroup of $G_0$.
\end{dfn}
If the group $G$ is finite and one considers only stem extensions, then there is a largest size for such a group $G_0$, and for every $G_0$ of that size the subgroup $H$ is isomorphic to the Schur multiplier of $G$. Moreover, if the finite group $G$ is perfect as well, then $G_0$ is unique up to isomorphism and is itself perfect. Such $G_0$ are often called \emph{universal perfect central extensions} of $G$, or \emph{covering groups}.

\begin{prop} \label{prop:B}
Let $H$ be a finite abelian group, and let $1 \to H \to G_0 \to G \to 1$ be a central extension of $G$ by $H$. Then either this extension is a stem extension, or $G_0$ has a non-trivial abelian quotient. \end{prop}
\begin{proof}
By definition, if the extension is not a stem extension, then $H \not\subseteq G'_0$, and thus $G_0/G'_0$ is a non-trivial abelian quotient.
\end{proof}

\begin{lem} \label {lem:A}
The Schur multiplier of $A_n$ is $C_2$ for $n=5$ or $n>7$ and it is $C_6$ for $n = 6$ or $7$.
\end{lem}
\begin{proof}
See 2.7 of \cite{Wilson-2009}
\end{proof}
\begin{lem} \label {lem:b}
The Schur multiplier of $\PSL_n(\F_{p^d})$ is a cyclic group of order $\mathrm{gcd}(n,p^d-1)$ except for $\PSL_2(\F_{4})$ (order 2), $\PSL_2(\F_{9})$ (order 6), $\PSL_3(\F_{2})$ (order 2), $\PSL_3(\F_{4})$ (order 48, product of cyclic groups of order 3, 4, 4) and $\PSL_4(\F_{2})$ (order 2).
\end{lem}
\begin{proof}
See 3.3 of \cite{Wilson-2009}.
\end{proof}
\subsection{Group extensions of groups with trivial centers}
Let $H$ and $F$ be groups, with $G$ a group extension of $H$ by $F$:
\begin{displaymath}
1 \to H \to G \to F \to 1
\end{displaymath}
Then, it is well known that $F$ acts on $H$ by conjugation, and this action induces a group homomorphism $\psi _G : F \to \mathrm{Out}\textrm{ }H$, which depends only on $G$.
\begin{lem}((7.11) of \cite{Suzuki-1982})
\label{lem:suzuki}
Suppose that $H$ has trivial center $(Z(H)=\{1\})$. Then, the structure of $G$ is uniquely determined by the homomorphism $\psi _G$. For any group homomorphism $\psi$ from $F$ to $\mathrm{Out}\textrm{ }H$, there exists an extension $G$ of $H$ by $F$ such that $\psi _G =\psi$. Moreover, the isomorphism class of $G$ is uniquely determined by $\psi$. (In particular, the class of $F \times H$ is determined by $\psi$ with $\psi(F)=1$.) All of the extensions are realized as a subgroup $U$ of the direct product $F \times \Aut\textrm{ }H$ satisfying the two conditions $U \cap \Aut\textrm{ }H = \mathrm{Inn}\textrm{ }H$ and $\pi(U) = F$, where $\pi$ is the projection from $F \times \Aut\textrm{ }H$ to $F$.
\end{lem}
\subsection{Prerequisites on $\GL_n(\F_q)$}
\subsubsection{General prerequisites}
The following lemma is well-known.
\begin{lem}\label{lem:schur}
Let $n\ge 2$, $q$ be a prime power, and let $U\le \GL_n(\F_q)$ act irreducibly on $(\F_q)^n$. Then the centralizer of $U$ in $\GL_n(\F_q)$ is cyclic.
\end{lem}
\begin{proof}
This follows immediately from Schur's lemma.
\end{proof}

\begin{lem}\label{lem:irred_cyclic}
Let $n\ge 2$, $q$ be a prime power and let $U\le \GL_n(\F_q)$ be cyclic, of order coprime to $q$. Assume that $U$ acts irreducibly on $(\F_q)^n$. Then the centralizer of $U$ in $\GL_n(\F_q)$ is cyclic of order $q^n-1$.
\end{lem}
\begin{proof}
This follows from \cite[Hilfssatz II.3.11]{Huppert}. Namely, setting $G:=C_{\GL_n(\F_q)}$, the centralizer of $U$ in $GL_n(\F_q)$, that theorem states that $G$ is isomorphic to $\GL_1(\F_{q^n})$, and thus in particular cyclic of order $q^n-1$.
\end{proof}

An important special case of the previous lemma is the following:
\begin{lem}\label{lem:singer}
Let $n\ge 2$, $q$ be a prime power and let $p$ be a primitive prime divisor of $q^n-1$, that is $p$ divides $q^n-1$, but does not divide any of the numbers $q^k-1$ with $1\le k<n$. Then the following hold:
\begin{itemize}
\item[i)] There is a unique non-trivial linear action of $C_p$ on $(\F_q)^n$, and this action is irreducible.
\item[ii)] The centralizer of a subgroup of order $p$ in $\GL_n(\F_q)$ is cyclic, of order $q^n-1$.
\end{itemize}
\end{lem}
\begin{proof}
Let $U<\GL_n(\F_q)$ be any subgroup isomorphic to $C_p$. From Maschke's theorem, it follows immediately that $U$ acts irreducibly on $(\F_n)^q$. From Lemma \ref{lem:irred_cyclic}, the centralizer of $U$ in $GL_n(\F_q)$ is then cyclic, of order $q^n-1$. Finally, every such $U$ is the unique subgroup of order $p$ of some $p$-Sylow subgroup of $\GL_n(\F_q)$ (note that, by assumption, the $p$-Sylow subgroups are of order dividing $q^n-1$, and then in fact cyclic, since $\GL_1(\F_{q^n})\le \GL_n(\F_q)$ is cyclic). Therefore all such subgroups $U$ are conjugate in $\GL_n(\F_q)$, proving the uniqueness in i).
\end{proof}
In the following sections, we collect some results about more specific linear groups.
\subsubsection{Structure of $\GL_2(\F_p)$}
\begin{lem} \label {lem:14}
$\GL_2(\F_p)$ does not contain any non-abelian simple subgroups for any prime $p$.
\end{lem}
\begin{proof}
Let $S$ be non-abelian simple. Then it is known that $S$ contains a non-cyclic abelian subgroup (see e.g.\ \cite[Corollary 6.6]{KLN}), and therefore even some subgroup $C_r\times C_r$ for some prime $r$.
On the other hand, as a direct consequence of Schur's lemma, any subgroup $C_r\times C_r$ of $\GL_2(\F_p)$ must intersect the center of $\GL_2(\F_p)$ non-trivially.\footnote{To apply Schur's lemma here, we have used that $p\ne r$, which is obvious, since $p^2$ does not divide $|GL_2(\F_p)|$.} Since $S$ has trivial center, it follows that $S$ cannot be contained in $\GL_2(\F_p)$.
\end{proof}
\subsubsection{Structure of $\GL_4(\F_2)$}
This article uses the structure of $\GL_4(\F_2)$. Thus, we recall several structural properties of this group.
\begin{prop} \label {prop:a}
$A_8$ is isomorphic to $\PSL_4(\F_2) = \GL_4(\F_2)$.
\end{prop}
\begin{lem} \label{lem:9}
$A_8$ does not contain a subgroup isomorphic to $A_5 \times C_2$ or $\SL_2(\F_5)$.
\end{lem}
\begin{proof}
Both $A_5 \times C_2$ and $\SL_2(\F_5)$ contain an element of order $10$, but there is no element of order $10$ in $A_8$.\end{proof}

\begin{lem} \label{lem:10}
The class of $(12345)$ is the unique conjugacy class of elements of order $5$ in $A_8$. In particular, there is a unique non-trivial linear $C_5$-action on $(\F_2)^4$. This action is irreducible.
\end{lem}
\begin{proof}
This is a special case of Lemma \ref{lem:singer}, with $q=2$ and $n=4$.
\end{proof}

\subsubsection{Structure of $\GL_4(\F_3)$}
We also make use of the structure of $\GL_4(\F_3)$ in this article. So we recall several structural properties of this group. We proved the following lemmas, partially aided by the computer program Magma.
\begin{lem} \label{lem:11}
$\GL_4(\F_3)$ contains a unique conjugacy class of subgroups isomorphic to $A_5 \times C_2$.
\end{lem}
\begin{proof}
By computer calculation, we can check that $\GL_4(\F_3)$ has four conjugacy classes of subgroups of order $120$. They are
\begin{equation}
\begin{split}
&\Bigg\langle\sm 0&0&2&0\\0&2&0&1&\\1&0&0&0\\0&1&0&1 \esm, \sm1&1&1&2\\2&0&0&2\\2&1&0&0\\1&0&2&0\esm,\sm2&0&0&0\\0&2&0&0\\0&0&2&0\\0&0&0&2\esm \Bigg\rangle, \Bigg\langle \sm 1&1&0&2\\0&2&0&0\\2&1&2&2\\0&0&0&2 \esm, \sm2&0&1&1\\1&2&0&1\\0&0&1&0\\0&0&1&2\esm\Bigg\rangle \\
&\Bigg\langle\sm 2&1&2&2\\2&0&1&1\\1&1&2&0\\1&1&0&2\esm,\sm 0&0&0&2\\2&1&0&2\\2&2&0&1\\0&2&2&0\esm,\sm 2&0&0&0\\0&2&0&0\\0&0&2&0\\0&0&0&2\esm\Bigg\rangle \textrm{ and } \Bigg\langle\sm2&2&0&1\\0&1&0&0\\1&2&1&1\\0&0&0&1 \esm, \sm1&0&2&2\\2&1&0&2\\0&0&2&0\\0&0&2&1\esm\Bigg\rangle.
\end{split}
\end{equation}
We use Magma to check that $\Big\langle\sm 2&1&2&2\\2&0&1&1\\1&1&2&0\\1&1&0&2\esm,\sm 0&0&0&2\\2&1&0&2\\2&2&0&1\\0&2&2&0\esm,\sm 2&0&0&0\\0&2&0&0\\0&0&2&0\\0&0&0&2\esm\Big\rangle$ is the only the conjugacy class of subgroup of order $120$ which is isomorphic to $A_5 \times C_2$.
\end{proof}
\begin{lem} \label{lem:12}
$\GL_4(\F_3)$ does not contain a subgroup isomorphic to $A_5 \times V_4$.
\end{lem}
\begin{proof}
$A_5\times V_4$ contains an abelian subgroup isomorphic to $C_{10}\times C_2$. 
As a special case of Lemma \ref{lem:singer} (with $q=3$, $n=4$), the centralizer of a cyclic group of order $5$ in $\GL_4(\F_3)$ is cyclic, of order $3^4-1=80$. Now of course, if $\GL_4(\F_3)$ contained a subgroup isomorphic to $C_{10}\times C_2$, then the centralizer of a respective subgroup of order $5$ would be non-cyclic. This ends the proof. 
\end{proof}

\begin{lem} \label{lem:13}
There exist a unique conjugacy class of elements of order $5$ in $GL_4(\F_3)$. Furthermore, there is a unique non-trivial linear action of $C_5$ on $(\F_3)^4$, and this action is irreducible.
\end{lem}
\begin{proof}
This again follows directly from Lemma \ref{lem:singer}, with $q=3$ and $n=4$.
\end{proof}

%
\subsubsection{Structure of $\GL_3(\F_5)$}
We will also use the structures of $\GL_3(\F_5)$.
5
\begin{lem} \label{lem:15}
$\GL_3(\F_5)$ contains a unique conjugacy class of subgroups isomorphic to $A_5 \times C_2$.
\end{lem}
\begin{proof}
By computer calculation, we can check that $\GL_4(\F_3)$ has four conjugacy classes of subgroups of order $120$. They are
\begin{equation}
\begin{split}
&\Big\langle\sm 2&1&2\\3&0&0\\2&3&4 \esm, \sm 0&1&1\\3&4&1\\4&2&1 \esm,\sm 1&3&2\\1&3&1\\1&4&0 \esm \Big\rangle,\Big\langle\sm 4&0&1\\0&4&0\\4&1&0 \esm, \sm 2&3&1\\3&0&3\\3&4&4 \esm \Big\rangle,\\
&\Big\langle\sm 1&0&0\\0&1&0\\1&4&4 \esm, \sm 1&0&4\\2&1&1\\3&0&3 \esm, \sm 4&0&0\\0&4&0\\0&0&4 \esm \Big\rangle, \textrm{ and } \Big\langle\sm 1&0&4\\0&1&0\\1&4&0 \esm, \sm 3&2&4\\2&0&2\\2&1&1 \esm \Big\rangle.
\end{split}
\end{equation}
We use Magma to check that $\Big\langle\sm 1&0&0\\0&1&0\\1&4&4 \esm, \sm 1&0&4\\2&1&1\\3&0&3 \esm, \sm 4&0&0\\0&4&0\\0&0&4 \esm \Big\rangle$ is the only the conjugacy class of subgroup of order $120$ which is isomorphic to $A_5 \times C_2$.
\end{proof}
\begin{lem} \label{lem:16}
$\GL_3(\F_5)$ does not contain a subgroup isomorphic to $A_5 \times V_4$.
\end{lem}
\begin{proof}
By Lemma \ref{lem:14}, any subgroup $A_5\le \GL_3(\F_5)$ has to act irreducibly. Since $A_5\times V_4$ has non-cyclic center, the claim now follows immediately from Lemma \ref{lem:schur}.
\end{proof}

\subsubsection{Structures of $\GL_5(\F_2)$ and $\GL_6(\F_2)$}
\begin{lem}
\label{lem:17}
$\GL_5(\F_2)$ does not contain a subgroup isomorphic to $PSL_2(8)$.
\end{lem}
\begin{proof}
The group $PSL_2(\F_8)=SL_2(\F_8)$ contains cyclic subgroups of order $\frac{8^2-1}{8-1}=9$. However, $\GL_5(\F_2)$ does not contain any such subgroups. Indeed, since $9$ is a prime power, Maschke's theorem implies that the existence of such a cyclic subgroup would enforce the existence of an irreducible cyclic subgroup of order $9$ in some $\GL_d(\F_2)$ with $d\le 5$. Then $2^d-1$ would have to be divisible by $9$, which is not the case for any such $d$. This concludes the proof.
\end{proof}

\begin{lem}
\label{lem:18}
$\GL_6(\F_2)$ contains a unique conjugacy class of subgroups isomorphic to $PSL_2(\F_8)$.
\end{lem}
\begin{proof}
Since $PSL_2(\F_8)=SL_2(\F_8)\le \GL_2(\F_8)$, the existence follows immediately from the well-known fact that $\GL_{n\cdot d}(\F_q)$ contains subgroups isomorphic to $\GL_{n}(\F_{q^d})$. The uniqueness can once again be verified with Magma.
\end{proof}

\begin{lem}
\label{lem:19}
$\GL_6(\F_2)$ does not contain subgroups isomorphic to $PSL_2(\F_8)\times C_2$.
\end{lem}
\begin{proof}
By Maschke's theorem (and using the proof of Lemma \ref{lem:17}), any cyclic subgroup of order $9$ in $\GL_6(\F_2)$ has to act irreducibly. By Lemma \ref{lem:irred_cyclic}, the centralizer of such a subgroup is then cyclic of order $2^6-1=63$. However, the centralizer of an order-$9$ subgroup in $PSL_2(\F_8)\times C_2$ is of course of even order. This concludes the proof.
\end{proof}
\section{Example: $K=\Q(\sqrt{22268})$}
\label{sec:ex4}
Let $K$ be the real quadratic number field $\Q(\sqrt{22268})$. We determine the Galois group of the maximal unramified extension of $K$.
\begin{thm}
\label{thm:22268}
Let $K$ be the real quadratic field $\Q(\sqrt{22268})$. Then, under the assumption of GRH, $\Gal(K_{ur}/K)$ is isomorphic to $A_5 \times C_2$. 
\end{thm}

The class number of $K$ is $2$, i.e., $\Cl(K) \simeq C_2$. Let $K_1$ be the Hilbert class field of $K$. Then $K_1$ can be written as $\Q(\sqrt{76},\sqrt{293})$. By computer calculation, we know that the class group of $K_1$ is trivial, i.e., $K_1$ has no nontrivial solvable unramified extensions.
\subsection{An unramified $A_5$-extension of $K_1$}
Let $K=\Q(\sqrt{22268})$ and let $L$ be the splitting field of
\begin{equation}
\begin{split} \label{poly4}
x^6 - 10x^4 - 7x^3 + 15x^2 + 14x + 3,
\end{split}
\end{equation}
a totally real polynomial with discriminant $19^2 \cdot 293^2$. We can also find the polynomial (\ref{poly4}) from the database of \cite{Kluners} and check that the discriminant of a root field of the polynimial(\ref{poly4}) is also $19^2 \cdot 293^2$. Then, $L$ is an $A_5$-extension over $\Q$ which is only ramified at $19$ and $293$. The  factorizations of the above polynomial modulo $19$ and $293$ are
\begin{displaymath}
x^6 - 10x^4 - 7x^3 + 15x^2 + 14x + 3 =(x+12)^2 (x+15)^2 (x^2+3 x+12)  \textrm{ mod } 19,
\end{displaymath}
\begin{displaymath}
x^6 - 10x^4 - 7x^3 + 15x^2 + 14x + 3 = (x+66)^2 (x+103) (x+160) (x+242)^2 \textrm{ mod } 293.
\end{displaymath}
Thus, $19$ and $293$ are the only primes ramified in this field with ramification index $2$. By Abhyankar's lemma, $LK_1/K_1$ is unramified at all primes, and $2$, $19$, and $293$ are the only primes ramified in $LK_1/\Q$ with ramification index $2$ (note that $22268=4 \cdot 19 \cdot 293$). Since $A_5$ is a nonabelian simple group, $L \cap K_1 =\Q$. Thus, $\Gal(LK_1/K_1) \simeq \Gal(L/\Q) \simeq A_5$, i.e., $LK_1$ is an unramified $A_5$-extension of $K_1$. We also know that $\Gal(LK_1/\Q) \simeq V_4 \times A_5$. Define $M$ as $LK_1$. 
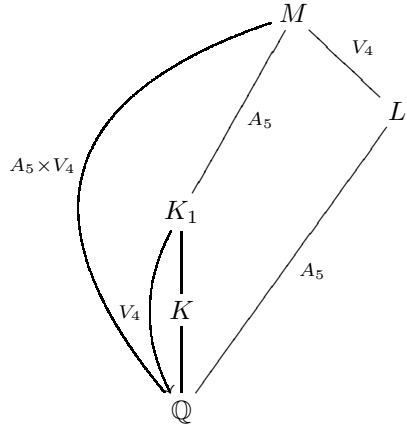
\[
\xymatrix{
& M \ar@{-}[ddl] ^-{A_5}\ar  @/_5pc/ @{-}[ddddl]_-{A_5 \times V_4} \ar@{-}[dr]^-{V_4} \\
& & L  \ar@{-}[dddll]^-{A_5} \\
K_1 \ar@{-}[d] \ar @/_1pc/ [dd]_{V_4} \\
K \ar@{-}[d] \\
\mathbb{Q}
}
\]

\subsection{Determination of $\Gal(K_{ur}/K)$}
To prove Theorem \ref{thm:22268}, it suffices to show that $M$ possesses no non-trivial unramified extensions.
Since $M/K$ is unramified, the root discriminant of $M$ is $|d_M|^{1/n_M}=|d_K|^{1/n_K}=\sqrt{22268}= 149.2246...$ If we assume GRH, then $|d_M|^{1/n_M}=|d_K|^{1/n_K}=\sqrt{22268}= 149.2246... <153.252 \leq B(31970,31970,0)$ (see the table in \cite{Martinet-1980}). This implies that $[K_{ur}:M]<\frac{31970}{[M:\Q]}=133.2083..$.

We now first exclude the existence of non-trivial unramified abelian extensions of $M$. Suppose $M$ possesses such an extension $T/M$. Without loss, $T/M$ can be assumed cyclic of prime degree. Let $T'$ be its normal closure over $\Q$. Then $T'$ is unramified and elementary-abelian over $K_1$, and $\Gal(M/K_1)\simeq A_5$ acts on $\Gal(T'/M)$.
The following intermediate result is useful.
\begin{lem}
\label{lem:no_central}
If $T/M$ is an unramified cyclic $C_p$-extension, then the action of $A_5$ on $\Gal(T'/M)$ is faithful or $[T':M]=2$.
\end{lem}
\begin{proof}
Since $A_5$ is simple, it suffices to exclude the case that the action of $A_5$ on $\Gal(T'/M)$ is trivial. In that case, the extension $1 \to \Gal(T'/M) \to \Gal(T'/K_1) \to A_5 \to 1$ would be a central extension. Assume that this extension is not a stem extension. In this case, $\Gal(T'/K_1)$ has a non-trivial abelian quotient by Proposition \ref{prop:B}. Since $T'/K_1$ is unramified, this contradicts the fact that $K_1$ has class number $1$. So the extension is a stem extension, whence Lemma \ref{lem:A} yields $\Gal(T'/M)\simeq C_2$.
\end{proof}
\begin{cor}
If $T/M$ is an unramified cyclic $C_p$-extension, then $\Gal(T'/M)$ is one of $(C_2)^{k}$ with $k\in \{1,4,5,6,7\}$, or $(C_3)^4$, or $(C_5)^3$.
\end{cor}
\begin{proof}
Lemma \ref{lem:no_central} shows that either $[T':M]=2$, or $A_5$ embeds into $\Aut(\Gal(T'/M))$. Furthermore, we already know $[T':M]\le 133$. Now it is easy to check that only the above possibilities for $\Gal(T'/M)$ remain (see in particular Lemma \ref{lem:14}).
\end{proof}
We now treat the remaining cases one by one.
\subsubsection{$2$-class group of $M$}
With the above notation, suppose that $\Gal(T/M)\simeq C_2$.
Then, $T'/M$ is unramified and $\Gal(T'/M)$ is isomorphic to $(C_2)^m$ ($1 \leq m \leq 7$).

Let $E \subset L $ be a root field of the polynomial (\ref{poly4}) and $N$ be the compositum of $E$ and $K_1$, i.e., $N=EK_1$. Then $E$ can be defined by the composite of three polynomials: $x^2-19$, $x^2-293$ and the polynomial (\ref{poly4}). By computer calculation, $N$ is a root field of the following polynomial: 
\begin{equation}
\begin{split}
&x^{24} - 3784x^{22} - 28x^{21} + 6404076x^{20} + 53312x^{19} - 6401641814x^{18} - \\
& 31411548x^{17} + 4204260566526x^{16} - 5837238288x^{15} - 1908791963697448x^{14} + \\
& 18501271313028x^{13} + 613640140988085895x^{12} - 11975084172112012x^{11} - \\
& 140616516271183965910x^{10} + 4264300576327196748x^9 + \\
&22779186389906647652933x^8 -  932994735936411884988x^7 - \\
&2542792801321996372912890x^6 +  124393633255686127917612x^5 + \\
&185598619641359536180924174x^4 - 9237397310199896463461164x^3 - \\
&7951324489796939270027088092x^2 + 291464252731787840722883096x +\\
&151174316045577424616769218057
\end{split}
\end{equation}
We also know that $\Gal(M/N)$ is isomorphic to $D_5$.
\[
\xymatrix{
M \ar@{-}[d] \ar @/^1pc/ @{-}[dd]^-{A_5 \times V_4} \ar @/_1pc/ @{-}[d]_-{D_5}\\
N \ar@{-}[d] \\
\mathbb{Q}
}
\]
By computer calculation, we know that the class group of $N$ is isomorphic to $C_2$ under GRH. Let $N'$ be the Hilbert class field of $N$. (Note that $N'$ is a subfield of $M$, since $M/N$ is unramified.)
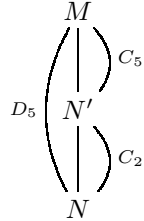
\[
\xymatrix{
M \ar@{-}[d] \ar @/^1pc/ @{-}[d]^-{C_5} \ar @/_1pc/ @{-}[dd]_-{D_5}\\
N' \ar@{-}[d] \ar @/^1pc/ @{-}[d]^-{C_2}  \\
N
}
\]
By Lemma \ref{2.6}, the $2$-class group of $N'$ is trivial. Thus the rank $m$ of the $2$-class group of $M$ is a multiple of $4$ by Lemma \ref{2.5}, i.e., $m$ is equal to $0$ or $4$.\\

Suppose that $m=4$. Then, $\Gal(T'/K_1)$ is an extension of $A_5$ by $(C_2)^4$. By Lemma \ref{lem:no_central},  $\Gal(M/K_1)$ acts faithfully on $\Gal(T'/M)$. Consider $\Gal(T'/K)$. This group is an extension of $\Gal(M/K) (\simeq A_5 \times C_2)$ by $\Gal(T'/M) (\simeq (C_2)^4)$ and an extension of $\Gal(K_1/K) (\simeq C_2)$ by $\Gal(T'/K_1)$ simultaneously. Therefore, it is natural to examine how $\Gal(K_1/K)$ acts on $\Gal(T'/M) (\simeq (C_2)^4)$. By Lemma \ref{lem:9}, $\Gal(M/K) (\simeq A_5 \times C_2)$ does not act faithfully on $\Gal(T'/M) (\simeq (C_2)^4)$. Since $\Gal(M/K_1) (\simeq A_5)$ acts nontrivially on $\Gal(T'/M)$, we obtain that $\Gal(K_1/K) (\simeq \Gal(M/LK))$ acts trivially on $\Gal(T'/M) (\simeq (C_2)^4)$.
\paragraph{$\Gal(T'/LK) \simeq (C_2)^5$}
Since $\Gal(M/LK)$ acts trivially on $\Gal(T'/M)$, $\Gal(T'/LK)$ is $(C_2)^3 \times C_4$ or $(C_2)^5$. Let $\Gal(T'/LK)$ be $(C_2)^3 \times C_4$. Then, $\Gal(T''/LK)$ is isomorphic to $(C_2)^4$, where $T''/LK$ is the maximal elementary abelian $2$-subextension of $T'/LK$. By the maximality of $T''$, $T''$ is also Galois over $\Q$ and $\Gal(T''/K)$ is an extension of $A_5$ by $(C_2)^4$. By restriction, this $A_5$-actions on $(C_2)^4$ comes from the $\Gal(M/K)$-actions on $\Gal(T'/M)$ mentioned above. Since $\Gal(T'/K_1)$ does not have any abelian quotient, $\Gal(T''/K)$ also has no abelian quotients, i.e., $T'' \cap K_1 =K$. Thus, $\Gal(T'/K)$ is a direct product of $\Gal(T''/K)$ and $\Gal(K_1/K)$, i.e., $\Gal(T'/LK)$ is a direct product of $\Gal(T''/LK) \simeq (C_2)^4$ and $\Gal(K_1/K)\simeq C_2$. This contradicts the fact that $\Gal(T'/LK)$ is $(C_2)^3 \times C_4$. Thus, $\Gal(T'/LK)$ is isomorphic to $(C_2)^5$, and there exists some $S/LK/K$ such that $SK_1=T'$ and $\Gal(S/K) \simeq (C_2)^4 \rtimes A_5$.

In a similar manner, we can prove that there exists some $S'/L/\Q$ such that $S'K_1=T'$ and $\Gal(S'/\Q) \simeq (C_2)^4 \rtimes A_5$.

Since $S'K$ is contained in $T'$, $S'K/K$ is an unramified extension. Therefore, the only ramified primes in $S'/L/\Q$ are $2$, $19$, and $293$ with ramification index $2$. Since $19$ and $293$ are already ramified in $L/\Q$, the only ramified prime in $S'/L$ is $2$.
\paragraph{Unramifiedness of $S'/L$} \label{unramifiedness}
Suppose that $2$ is ramified in $S'/L$. The ramification index of $2$ should then be $2$. Let $\bar{\p}$ (resp. $\p$) be a prime ideal in $S'$ (resp. $L$) satisfying $\bar{\p} | 2$ (resp. $\p|2$). The  factorization of the polynomial (\ref{poly4}) modulo $2$ is
\begin{equation}
\begin{split}
x^6 - 10x^4 - 7x^3 + 15x^2 + 14x + 3 \equiv (x+1) (x^5+x^4+x^3+x+1) \textrm{ mod }2.
\end{split}
\end{equation}
Thus, we know that $\Gal(L_{\p}/\Q_2)$ is isomorphic to $C_5 \simeq \langle (12345) \rangle$, where $L_{\p}$ is the $\p$-completion of $L$. Consider $\Gal(S'_{\bar{\p}}/L_{\p})$. Since the ramification index of $\p$ is $2$, $\Gal(S'_{\bar{\p}}/L_{\p})$ is $C_2$ or $(C_2)^2$, i.e., the proper subgroup of $(C_2)^4$. Hence, $\Gal(S'_{\bar{\p}}/\Q_3) = \Gal(S'_{\bar{\p}}/L_{\p}) \rtimes \langle (12345) \rangle \subsetneq (C_2)^4 \rtimes \langle (12345) \rangle$. This contradicts the statement that there is no proper subgroup of $(C_2)^4$ that is invariant under the action of $\langle (12345) \rangle$ (see Lemma \ref{lem:10}). Thus, $S'/L$ should be unramified at all places. In conclusion, $S'/\Q$ is a $(C_2)^4 \rtimes A_5$-extension of $\Q$ that has ramification index $2$ at only $19$ and $293$. Let us now consider the root discriminant of $S'$. Since $S'/L$ is unramified at all places,
\begin{displaymath}
|d_{S'}|^{1/n_{S'}}=|d_{L}|^{1/n_{L}}=(19^{30} \cdot 293^{30})^{1/60}= \sqrt{19\cdot 293} =74.6123......
\end{displaymath}
This implies that $|d_{S'}|^{1/n_{S'}} < 106.815..... \leq B(960,960,0)$ under the GRH (see the table in \cite{Martinet-1980}). This contradicts the definition of the lower bound for the root discriminant. Thus, the $2$-class group of $M$ is trivial.
\subsubsection{$3$-class group of $M$}\label{3-class}
Suppose that $T/M$ is an unramified $C_3$-extension. Then, as seen above, $T'$ is unramified over $M$ and $\Gal(T'/M)$ is isomorphic to $(C_3)^4$. 
Then, $\Gal(T'/\Q)$ is an extension of $\Gal(M/\Q) \simeq A_5 \times V_4$ by $(C_3)^4$. Therefore, it is natural to examine how $\Gal(M/\Q)$ acts on $\Gal(T'/M)\simeq (C_3)^4$. By Lemma \ref{lem:11} and Lemma \ref{lem:12}, we know that there are three possibilities of the actions of $\Gal(M/\Q)$ on $\Gal(T'/M)$. (Note that $\Aut((C_3)^4) \simeq \GL_4(\F_3)$). Each action is induced by the following three group homomorphisms $\psi : A_5 \times V_4 \to \GL_4(\F_3)$:\\

-$\psi$ is trivial.

-$\psi(A_5 \times V_4)\simeq A_5$.

-$\psi(A_5 \times V_4)\simeq A_5 \times C_2$.\\\\
By Lemma \ref{lem:no_central}, $\Gal(M/K_1)$ acts faithfully on $\Gal(T'/M)$. Therefore, $\psi$ cannot be trivial.
\paragraph{$\psi(A_5 \times V_4)\simeq A_5$} \label{5.2.2.1}
This means that $\Gal(M/K_1)(\simeq A_5)$ acts nontrivially on $\Gal(T'/M)$ and $\Gal(M/L) \simeq V_4$ acts trivially on $\Gal(T'/M)$. Since $|\Gal(T'/M)|$ and $|\Gal(M/L)|$ are coprime, $\Gal(T'/L)$ is isomorphic to $V_4 \times (C_3)^4$. Let $S$ be the subfield of $T'$ fixed by $V_4$. Then $\Gal(S/\Q)$ is a group extension of $A_5$ by $(C_3)^4$.
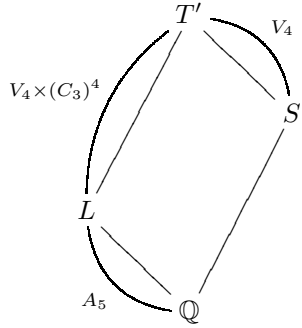
\[
\xymatrix{
& T'\ar@{-}[ddl]  \ar@{-}[dr] \ar @/_1pc/  @{-}[ddl]_{V_4 \times (C_3)^4} \ar @/^1pc/ @{-}[dr]^-{V_4} \\
& & S \ar@{-}[ddl] \\
L\ar@{-}[dr] \ar @/_1pc/  @{-}[dr]_{A_5}\\
&\mathbb{Q}
}
\]
 Since $19$ and $293$ are already ramified in $L/\Q$, the only ramified prime in $S/L$ is $2$. If $2$ is ramified in $S/L$, its ramification index should be $2$. But it is impossible, because the degree of $[S:L]$ is odd. Thus $S/L$ is unramified over all places. By a similar argument as in \ref{unramifiedness}, we can check that this contradicts the definition of the lower bound for the root discriminant.
\paragraph{$\psi(A_5 \times V_4)\simeq A_5 \times C_2$}
\label{3class_homomorphisms}
First of all, let us see the intermediate fields in $M/L$. Since $\Gal(M/L)$ is isomorphic to $V_4$, there are three proper intermediate fields in $M/L$.
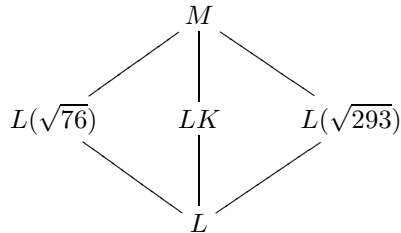
\[
\xymatrix{
&M \ar@{-}[dl] \ar@{-}[dr] \ar@{-}[d] \\
L(\sqrt{76}) \ar@{-}[dr] & LK\ar@{-}[d]& L(\sqrt{293}) \ar@{-}[dl]\\
& L
}
\]

Suppose that $\Gal(M/L(\sqrt{76}))$ acts trivially on $\Gal(T'/M)$. This means that $\Gal(T'/L(\sqrt{76}))$ is isomorphic to $C_2 \times (C_3)^4$, i.e., there exists a subfield $S$ in $T'/L(\sqrt{76})$ such that $\Gal(S/L(\sqrt{76}))$ is isomorphic to $(C_3)^4$.
\[
\xymatrix{
S  \ar@{-}[d] \ar @/^1pc/ @{-}[d]^-{(C_3)^4}\\
L(\sqrt{76}) \ar@{-}[d] \ar @/^1pc/ @{-}[d]^-{A_5}\\
\Q(\sqrt{76}) \ar@{-}[d] \ar @/^1pc/ @{-}[d]^-{C_2}\\
\Q
}
\]
We easily check that $S/L(\sqrt{76})$ is unramified over all places. Let $\bar{\p}$ (resp. $\p'$, $\p$) be a prime ideal in $S$ (resp. $L(\sqrt{76})$, $\Q(\sqrt{76})$) satisfying $\bar{\p} | 2$ (resp.$\p'|2$, $\p|2$). We had already show that the factorization of the polynomial (\ref{poly4}) modulo $2$ is
\begin{equation}
\begin{split}
x^6 - 10x^4 - 7x^3 + 15x^2 + 14x + 3 \equiv (x+1) (x^5+x^4+x^3+x+1) \textrm{ mod }2.
\end{split}
\end{equation}
Thus, we know that $\Gal(L(\sqrt{76})_{\p'}/\Q(\sqrt{76})_{\p})$ is isomorphic to $C_5 \simeq \langle (12345) \rangle$, where $L(\sqrt{76})_{\p'}$ (resp. $\Q(\sqrt{76})_{\p}$) is the $\p'$-completion of $L(\sqrt{76})$ (resp. the $\p$-completion of $\Q(\sqrt{76})_{\p}$). 

Let us consider $\Gal(S_{\bar{\p}}/L(\sqrt{76})_{\p'})$. We know that $S/L(\sqrt{76})$ is unramified. Thus, $S_{\bar{\p}}/L(\sqrt{76})_{\p'}$ is a cyclic extension, i.e., $\Gal(S_{\bar{\p}}/L(\sqrt{76})_{\p'})$ is isomorphic to $C_3$ or a trivial group.

Suppose that $\Gal(S_{\bar{\p}}/L(\sqrt{76})_{\p'})$ is isomorphic to $C_3$. Then $\Gal(S_{\bar{\p}}/\Q(\sqrt{76})_{\p})=\Gal(S_{\bar{\p}}/L(\sqrt{76})_{\p'}) \rtimes \langle (12345) \rangle \subsetneq (C_3)^4 \rtimes \langle (12345) \rangle$. This contradicts the statement that there is no proper subgroup of $(C_3)^4$ that is invariant under the action of $\langle (12345) \rangle$ (See Lemma \ref{lem:13}). In conclusion, $\Gal(S_{\bar{\p}}/L(\sqrt{76})_{\p'})$ is trivial. 

Thus, for a number field $S/\Q$, $e_2=2$ and $f_2=5$ where $e_2$ is the ramification index of $2$ and $f_2$ is the inertia degree for $2$. Let us recall the function (\ref{f})
\begin{displaymath}
f=2\sum_{\p}\sum^{\infty}_{m=1}\frac{\log N(\p)}{N(\p)^{m/2}}F(\log N(\p)^m).
\end{displaymath}
Since every term of $f$ is greater than or equal to $0$, the following holds for the number field $S$.
\begin{equation}
\begin{split}
f \geq 2\sum^{972}_{j=1}\sum^{100}_{i=1}\frac{\log N(\bar{\q}_j)}{N(\bar{\q}_j)^{i/2}}F(\log N(\bar{\q}_j)^i),
\end{split}
\end{equation}
where the $\bar{\q}_j$ denote the prime ideals of $S$ satisfying $\bar{\q}_j | 2$. Since $f_2=5$, $N(\bar{\q}_j)=2^5$ for all $j$. Set $b=8.8$. By a numerical calculation, we have
\begin{equation}
\begin{split}
f \geq 2\cdot 972 \sum^{100}_{i=1}\frac{\log 2^5}{2^{5i/2}}F(\log 2^{5i})=1111.46....
\end{split}
\end{equation}
Let us recall (\ref{D}). For $b=8.8$, we have
\begin{equation}
\begin{split}
|d_{S}|^{1/n_{S}}&>149.272\cdot e^{(f-604.89)/9720}\\
&\geq 149.272\cdot e^{(1111.46-604.89)/9720}=157.258....
\end{split}
\end{equation}
$|d_{S}|^{1/n_{S}}=|d_K|^{1/n_K}=\sqrt{22268}$ contradicts the fact that $|d_{S}|^{1/n_{S}}= 149.2246...$

Next, suppose that $\Gal(M/LK)$ acts trivially on $\Gal(T'/M)$. This means that $\Gal(T'/LK)$ is isomorphic to $C_2 \times (C_3)^4$, i.e., there exists a subfield $S'$ in $T'/LK$ such that $\Gal(S'/LK)$ is isomorphic to $(C_3)^4$.
\[
\xymatrix{
S'  \ar@{-}[d] \ar @/^1pc/ @{-}[d]^-{(C_3)^4}\\
LK \ar@{-}[d] \ar @/^1pc/ @{-}[d]^-{A_5}\\
K \ar@{-}[d] \ar @/^1pc/ @{-}[d]^-{C_2}\\
\Q
}
\]
By the same argument as in the above, we can get
\begin{equation}
\begin{split}
|d_{S'}|^{1/n_{S}}&>157.258....
\end{split}
\end{equation}
and this contradicts the fact that $|d_{S}|^{1/n_{S}}= 149.2246...$.

Finally, suppose that $\Gal(M/L(\sqrt{293}))$ acts trivially on $\Gal(T'/M)$. This means that $\Gal(T'/L(\sqrt{293}))$ is isomorphic to $C_2 \times (C_3)^4$, i.e., there exists a subfield $S''$ in $T'/L(\sqrt{293})$ such that $\Gal(S''/L(\sqrt{293}))$ is isomorphic to $(C_3)^4$.
\[
\xymatrix{
S''  \ar@{-}[d] \ar @/^1pc/ @{-}[d]^-{(C_3)^4}\\
L(\sqrt{293}) \ar@{-}[d] \ar @/^1pc/ @{-}[d]^-{A_5}\\
\Q(\sqrt{293}) \ar@{-}[d] \ar @/^1pc/ @{-}[d]^-{C_2}\\
\Q
}
\]
We easily know that $19$ and $293$ are the only ramified primes in $S''/\Q$. By a similar argument as in section \ref{5.2.2.1}, we can check that this contradicts the definition of the lower bound for the root discriminant.\\

In conclusion, the $3$-class group of $M$ is trivial.
\subsubsection{$5$-class group of $M$}
Suppose that $T/M$ is an unramified $C_5$-extension. Then, $T'$ is unramified over $M$ and $\Gal(T'/M)$ is isomorphic to $(C_5)^3$. Thus, $\Gal(T'/\Q)$ is an extension of $\Gal(M/\Q) \simeq A_5 \times V_4$ by $(C_5)^3$. Therefore, it is natural to examine how $\Gal(M/\Q)$ acts on $\Gal(T'/M)\simeq (C_5)^3$. By Lemma \ref{lem:15} and Lemma \ref{lem:16}, we know that there are three possibilities of the actions of $\Gal(M/\Q)$ on $\Gal(T'/M)$. Each action is induced by the following three group homomorphisms $\psi : A_5 \times V_4 \to \GL_3(\F_5)$.:\\

-$\psi$ is trivial.

-$\psi(A_5 \times V_4)\simeq A_5$.

-$\psi(A_5 \times V_4)\simeq A_5 \times C_2$.\\\\
By a similar argument as in section \ref{3-class}, we just need to think about the case $\psi(A_5 \times V_4)\simeq A_5 \times C_2$.
\paragraph{$\psi(A_5 \times V_4)\simeq A_5 \times C_2$}
Consider again the intermediate fields of $M/L$ as in \S \ref{3class_homomorphisms}.
Suppose that $\Gal(M/L(\sqrt{76}))$ acts trivially on $\Gal(T'/M)$. This means that $\Gal(T'/L(\sqrt{76}))$ is isomorphic to $C_2 \times (C_5)^3$, i.e., there exists a subfield $S$ in $T'/L(\sqrt{76})$ such that $\Gal(S/L(\sqrt{76}))$ is isomorphic to $(C_5)^3$.
\[
\xymatrix{
S  \ar@{-}[d] \ar @/^1pc/ @{-}[d]^-{(C_5)^3}\\
L(\sqrt{76}) \ar@{-}[d] \ar @/^1pc/ @{-}[d]^-{A_5}\\
\Q(\sqrt{76}) \ar@{-}[d] \ar @/^1pc/ @{-}[d]^-{C_2}\\
\Q
}
\]
From \cite{Kluners}, we know that $L$ can also be defined as the splitting field of following polynomial, corresponding to an imprimitive degree-$12$ action of $A_5$:
\begin{equation}\label{poly12}
\begin{split}
& x^{12} + 11x^{11} - 59x^{10} - 647x^9 - 295x^8 + 5446x^7 + 4294x^6 -\\
&  14727x^5 - 4960x^4 + 16477x^3 - 4028x^2 - 1813x + 324.
\end{split}
\end{equation}
Let $E\subset L$ be a root field of the polynomial $(\ref{poly12})$. We know that the discriminant $d_E$ of $E$ is $19^6 \cdot 293^6$. Since $|d_E|^{1/n_E}=|d_L|^{1/n_L}$, $L/E$ is unramified.

Define $N$ as the compositum of $E$ and $\Q(\sqrt{76})$. Then $N$ is a subfield of $L(\sqrt{76})$ and $\Gal(L(\sqrt{76})/N)$ is isomorphic to $C_5$.
\[
\xymatrix{
S  \ar@{-}[d] \ar @/^1pc/ @{-}[d]^-{(C_5)^3}\\
L(\sqrt{76}) \ar@{-}[d] \ar @/^1pc/ @{-}[d]^-{C_5}\\
N
}
\]
By Abhyankar's lemma, we easily know that $L(\sqrt{76})/N$ is unramified. Using a computer calculation, we can check that $N$ is a root field of the following polynomial:
\begin{equation}\label{poly24-1}
\begin{split}
& x^{24} - 111x^{22} + 4394x^{20} - 83286x^{18} + 818659x^{16} - 4122356x^{14} + 9878557x^{12} - \\
&10688099x^{10} + 5561624x^8 - 1360039x^6 + 130854x^4 - 2499x^2 + 1.
\end{split}
\end{equation}

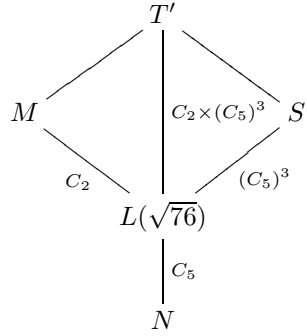
\[
\xymatrix{
&T' \ar@{-}[dl]  \ar@{-}[dr] \ar@{-}[dd]^-{C_2 \times (C_5)^3}\\
M \ar@{-}[dr]_{C_2} & & S \ar@{-}[dl]^-{(C_5)^3}\\
& L(\sqrt{76})\ar@{-}[d]^-{C_5}\\
& N
}
\]
By the calculation of sage, we can check that the class group of $N$ is equal to $C_{10}$, i.e., $5$-class group of $N$ is $C_5$ and Hilbert $5$-class field of $N$ is $L(\sqrt{76})$.  We know that $\Gal(T'/L(\sqrt{76}))$ is isomorphic to $C_2 \times (C_3)^5$ i.e., $5$-class group of $L(\sqrt{76})$ is not trivial. This contradicts Lemma \ref{2.6}.\\

Suppose that $\Gal(M/LK)$ acts trivially on $\Gal(T'/M)$. Define $N'$ as the compositum of $E$ and $K$. Then $N$ can be defined by the following polynomial:
\begin{equation}\label{poly24-2}
\begin{split}
& x^{24} - 98x^{22} + 4073x^{20} - 94476x^{18} + 1354898x^{16} - 12553566x^{14} + \\
&76075696x^{12} - 297782263x^{10} + 723063287x^8 - 1000608193x^6 + \\
&654400814x^4 - 110097135x^2 + 3818116.
\end{split}
\end{equation}
By a computer calculation with Magma, we can check, assuming GRH, that the class group of $N$ is equal to $C_{10}$, i.e., the $5$-class group of $N$ is $C_5$ and the Hilbert $5$-class field of $N$ is $LK$. By the same argument as above, we obtain a contradiction. \\

Suppose that $\Gal(M/L(\sqrt{293}))$ acts trivially on $\Gal(T'/M)$. This means that $\Gal(T'/L(\sqrt{293}))$ is isomorphic to $C_2 \times (C_5)^3$, i.e., there exists a subfield $S''$ in $T'/L(\sqrt{293})$ such that $\Gal(S''/L(\sqrt{293}))$ is isomorphic to $(C_5)^3$, and such that $19$ and $293$ are the only ramified primes in $S''/\Q$.
By a similar argument as in section \ref{5.2.2.1}, we can check that this contradicts the lower bound for the root discriminant.

In conclusion, $5$-class group of $M$ is also trivial under the assumption of the GRH. We have therefore obtained:
\begin{prop}
The class number of $M$ is $1$, under the assumption of the GRH.
\end{prop}
\subsubsection{$A_5$-unramified extension of $M$}\label{A5}
Since the class number of $M$ is one, there is no solvable unramified extension over $M$. The last thing we have to do is to show that there is no nonsolvable unramified extension over $M$. Since $[K_{ur}:M]<133.2083..$, our task is to show that K does not admit an unramified $A_5$-extension.

Suppose that $M$ admits an unramified $A_5$-extension $F$. Because $[K_{ur}:M]<134$, $F$ is the unique unramified $A_5$-extension of $M$, i.e., $F$ is Galois over $\Q$. It is well known that $A_5$ is isomorphic to $\PSL_2(\F_5)$ and $S_5$ is isomorphic to $\PGL_2(\F_5)$. By Lemma \ref{lem:suzuki}, $\Gal(F/K_1) \simeq A_5 \times A_5$, i.e., $K_1$ admits another $A_5$-unramified extension $F_1$. 
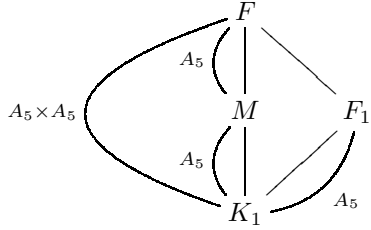
\[
\xymatrix{
F \ar@{-}[d] \ar @/_1pc/ @{-}[d]_-{A_5} \ar @/_5pc/ @{-}[dd]_-{A_5 \times A_5} \ar@{-}[dr]\\
M \ar@{-}[d] \ar @/_1pc/ @{-}[d]_-{A_5} & F_1 \ar@{-}[dl] \ar @/^1pc/ @{-}[dl]^-{A_5}\\
K_1
}
\]
(Note that $F_1$ is also Galois over $\Q$, or otherwise $K_1$ would have further unramified $A_5$-extensions, contradicting Odlyzko's bound.) Then, by Lemma \ref{lem:suzuki}, there are only two possibilities for $\Gal(F_1/K)$: $A_5 \times C_2$ or $S_5$.

\paragraph{Case 1 - $\Gal(F_1/K) \simeq A_5 \times C_2$ }
By a similar argument in the above, $K$ admits an $A_5$-unramified extension $F_2$. Then, $\Gal(F_2/\Q)$ is also isomorphic to $A_5 \times C_2$ or $S_5$.
\paragraph{Case 1.1 - $\Gal(F_2/\Q) \simeq A_5 \times C_2$ }
\label{case11}
This implies that there exists an $A_5$-extension $F_3/\Q$ with all ramification indices $\le 2$ and unramified outside of $\{2,19,293\}$. Assume first that $19$ is unramified in $F_3/\Q$. Let $E$ be a quintic subfield of $F_3/\Q$. Then, by a well-known result of Dedekind, we get the upper bound $|d_E|\le 2^6\cdot 293^2 < 5.5\cdot 10^{6}$ for the discriminant of $E$. However, from Table 2 in \cite[Section 4.1]{Schwarz-1994} no extension with this discriminant bound and ramification restrictions exists. We may therefore assume that $19$ is ramified in $F_3/\Q$. Since its inertia group is generated by a double transposition in $A_5$, the inertia degree of $19$ in the extension $F_2/\Q$ (with Galois group $A_5 \times C_2$) is at most $2$. The same holds for the inertia degree of $19$ in the extension $L/\Q$, and therefore eventually also in the compositum $LF_2/\Q$.

Let us recall the function (\ref{f})
\begin{displaymath}
f=2\sum_{\p}\sum^{\infty}_{m=1}\frac{\log N(\p)}{N(\p)^{m/2}}F(\log N(\p)^m).
\end{displaymath}
Since every term of $f$ is greater than or equal to $0$, the following holds for the number field $LF_2$.
\begin{equation}
\begin{split}
f \geq 2\sum^{1800}_{j=1}\sum^{100}_{i=1}\frac{\log N(\bar{\q}_j)}{N(\bar{\q}_j)^{i/2}}F(\log N(\bar{\q}_j)^i),
\end{split}
\end{equation}
where the $\bar{\q}_j$ denote the prime ideals of $LF_2$ satisfying $\bar{\q}_j | 19$. Since $f_{19}=2$, $N(\bar{\q}_j)=19^2$ for all $j$. Set $b=8.8$. By a numerical calculation, we have
\begin{equation}
\begin{split}
f \geq 2\cdot 1800 \sum^{100}_{i=1}\frac{\log 19^2}{19^{i}}F(\log 19^{2i})=683.225....
\end{split}
\end{equation}
Let us recall (\ref{D}). For $b=8.8$, we have
\begin{equation}
\begin{split}
|d_{LF_2}|^{1/n_{LF_2}}&>149.272\cdot e^{(f-604.89)/7200}\\
&\geq 149.272\cdot e^{(683.225-604.89)/7200}=150.905....
\end{split}
\end{equation}
$|d_{LF_2}|^{1/n_{LF_2}}=|d_K|^{1/n_K}=\sqrt{22268}$ contradicts the fact that $|d_{LF_2}|^{1/n_{LF_2}}= 149.2246...$

\paragraph{Case 1.2 - $\Gal(F_2/\Q) \simeq S_5$ }
By the unramifiedness of $F_2/K$, and since the only involutions of $S_5$ not contained in $A_5$ are the transpositions, a quintic subfield $E$ of $F_2$ must have the discriminant $22268$. However,  such a quintic number field does not exist, from \cite{Schwarz-1994}. This is a contradiction.
\paragraph{Case 2 - $\Gal(F_1/K) \simeq S_5$ }
By Lemma \ref{lem:suzuki}, $\Gal(F_1/\Q) \simeq S_5 \times C_2$. Consequently, $F_1$ is the compositum of $K$ and an $S_5$-extension $F_2$ of $\Q$. 
Furthermore, $F_2/\Q$ has a quadratic subextension contained in $K_1$, but linearly disjoint from $K$. Therefore it is either $\Q(\sqrt{293})$ or $\Q(\sqrt{76})$. Consider now a quintic subfield $E$ of $F_2/\Q$. Of course, $E/\Q$ is unramified outside $\{2,19,293\}$. Furthermore, all non-trivial inertia subgroups are generated either by transpositions or by double transpositions. Finally, the inertia subgroups at those primes which ramify in the quadratic subfield of $F_2/\Q$ are generated by transpositions. By a similar argument as in \S\ref{case11}, we then get one of the following two upper bounds for the discriminant of $E$: Either $|d_E| \le 2^3\cdot 19 \cdot 293^2$ (namely, if the quadratic subfield is $\Q(\sqrt{76})$), or $|d_E| \le 2^6\cdot 19^2 \cdot 293$.
 Such a quintic number field does not exist, from \cite[Section 4.1]{Schwarz-1994}. This is a contradiction.

In conclusion, $M$ admits no unramified $A_5$-extensions, i.e., we have that $\Gal(K_{ur}/K_1) \cong A_5$ under the assumption that the GRH holds. This concludes the proof of Theorem \ref{thm:22268}.
\section{Appendix: $K=\Q(\sqrt{-1567})$}
Until now, we dealt with real quadratic fields. In this section, we will give the first case of an imaginary quadratic field. 	

Let $K$ be the imaginary quadratic number field $\Q(\sqrt{-1567})$. We show the following:
\begin{thm}
\label{thm:1567}
Let $K$ be the imaginary quadratic field $\Q(\sqrt{-1567})$ and $K_{ur}$ be its maximal unramified extension Then $\Gal(K_{ur}/K)$ is isomorphic to $\PSL_2(\F_8) \times C_{15}$ under the assumption of the GRH..
\end{thm}
The class number of $K$ is $15$, i.e., $\Cl(K) \simeq C_{15}$. Let $K_1$ be the Hilbert class field of $K$.
\subsection{Class number of $K_1$}
The first thing we have to do is show that the class number of $K_1$ is one. 
It can be computed that $K_1$ is the splitting field of the polynomial
\begin{equation}
\begin{split}
&x^{15} + 14x^{14} + 56x^{13} + 105x^{12} + 497x^{11} + 832x^{10} + 1157x^9 + 1274x^8 +\\
& 644x^7 - 971x^6 - 2582x^5 - 177x^4 + 7x^3 + 1187x^2 - 20x + 1
\end{split}
\end{equation}
We can then check with Magma that the class number of $K_1$ is $1$, under GRH.
\subsection{An unramified $\PSL_2(\F_8)$-extension of $K_1$}
Let $K=\Q(\sqrt{-1567})$ and let $L$ be the splitting field of
\begin{equation}
\begin{split} \label{poly3}
x^9 - 2x^8 + 10x^7 - 25x^6 + 34x^5 - 40x^4 + 52x^3 - 45x^2 + 20x - 4,
\end{split}
\end{equation}
a polynomial with complex roots. Then $L$ is a $\PSL_2(\F_8)$-extension of $\Q$ and $1567$ is the only prime ramified in this field with ramification index two. By Abhyankar's lemma, $LK/K$ is unramified at all primes. Since $\PSL_2(\F_8)$ is a nonabelian simple group, $L \cap K_1 =\Q$. So $\Gal(LK_1/K_1) \simeq \Gal(L/\Q) \simeq \PSL_2(\F_8)$, i.e., $LK_1$ is a $\PSL_2(\F_8)$-extension of $K_1$ which is unramified over all places. It follows that $\Gal(LK_1/\Q)$ is isomorphic to $\PSL_2(\F_8) \times D_{15}$.
\[
\xymatrix{
& K_1 L \ar@{-}[ddl] ^-{\PSL_2(\F_8)}\ar  @/_5pc/ @{-}[ddddl]_-{\PSL_2(\F_8) \times D_{15}} \ar@{-}[dr]^-{D_{15}} \\
& & L  \ar@{-}[dddll]^-{\PSL_2(\F_8)} \\
K_1 \ar@{-}[d] \ar @/_1pc/ [dd]_{D_{15}} \\
K \ar@{-}[d] \\
\mathbb{Q}
}
\]
\subsection{The determination of $\Gal(K_{ur}/K)$}
Define $M$ as $LK_1$. Since $M/K$ is unramified at all places, the root discriminant of $M$ is $|d_{M}|^{1/|M|}=|d_{K}|^{1/{|K|}}=\sqrt{1567}= 39.5853...$.  If we assume GRH, then $|d_{M}|^{1/|M|}=|d_{K}|^{1/{|K|}}=\sqrt{1567}= 39.5853<39.895...= B(1000000,0,500000)$. (See the table in \cite{Martinet-1980}). This imply that $[K_{ur}:M]<\frac{1000000}{[M:\Q]}=66.1375...$. We now proceed similarly as in Section \ref{sec:ex4}. Let $T$ be a non-trivial unramified $C_p$-extension of $M$, and let $T'$ be its Galois closure over $\Q$ First, we obtain the following analog of Lemma \ref{lem:no_central}.
\begin{lem}
\label{lem:no_central2}
If $T/M$ is a non-trivial unramified cyclic $C_p$-extension, then the action of $\PSL_2(\F_8)$ on $\Gal(T'/M)$ is faithful.
\end{lem}
\begin{proof}
As in Lemma \ref{lem:no_central}, and using additionally that $\PSL_2(\F_8)$ has trivial Schur multiplier (see Lemma \ref{lem:b}).
\end{proof}
\begin{cor}
If $T/M$ is a non-trivial unramified cyclic $C_p$-extension, then $p=2$ and $\Gal(T'/M)\simeq (C_2)^6$.
\end{cor}
\begin{proof}
Use Lemma \ref{lem:no_central2}, the bound $[T':M]\le 66$, and Lemmata \ref{lem:14} and \ref{lem:17} in order to obtain that $(C_2)^6$ is the only elementary-abelian group in the relevant range which allows a non-trivial $\PSL_2(\F_8)$-action. 
\end{proof}
We deal with the remaining case below.
\subsubsection{$2$-class group of $M$}
Suppose that $M$ has an unramified $C_2$-extension $T$ and let $T'$ be its normal closure over $\Q$. As shown above, $T'$ is unramified over $M$ and $\Gal(T'/M)$ is isomorphic to $(C_2)^6$.
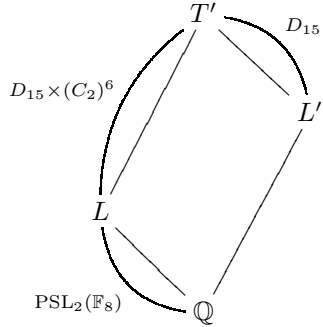
\[
\xymatrix{
& T'\ar@{-}[ddl]  \ar@{-}[dr] \ar @/_1pc/  @{-}[ddl]_{D_{15} \times (C_2)^6} \ar @/^1pc/ @{-}[dr]^-{D_{15}} \\
& & L' \ar@{-}[ddl] \\
L\ar@{-}[dr] \ar @/_1pc/  @{-}[dr]_{\PSL_2(\F_8)}\\
&\mathbb{Q}
}
\]
Let $\bar{\p}$ (resp. $\p$) be a prime ideal in $L'$ (resp. $L$) satisfying $\bar{\p} | 2$ (resp. $\p|2$). The  factorization of the polynomial (\ref{poly3}) modulo $2$ is
\begin{equation}
\begin{split}
x^2 (x^7 + x^4 + 1) \textrm{ mod }2.
\end{split}
\end{equation}
Since $\PSL_2(\F_8)$ contains no elements of order $14$, we thus know that $\Gal(L_{\p}/\Q_2)$ is isomorphic to $C_7$, where $L_{\p}$ is the $\p$-completion of $L$.
Consider $\Gal(L'_{\bar{\p}}/L_{\p})$. Because $L'/L$ is unramified, $\Gal(L'_{\bar{\p}}/L_{\p})$ is either trivial or $C_2$. 
\[
\xymatrix{
L' \ar@{-}[d] \ar @/^1pc/ @{-}[d]^-{(C_2)^6}\\
L 
}
\]
By Lemma \ref{lem:18}, there is a unique class of subgroups $\PSL_2(\F_8)$ inside $\GL_6(\F_2)$. The cyclic subgroups of order $7$ in these subgroups act fixed-point-freely on $(C_2)^6$ (in fact, the vector space decomposes into a direct sum of two irreducible modules of dimension $3$ under their action). Therefore, the corresponding group extension of $C_7$ by $(C_2)^6$ has trivial center, and in particular contains no element of order $14$.
Thus, $\Gal(L'_{\bar{\p}}/L_{\p})$ is trivial, i.e., $\p$ splits completely in $L'$. \\

Define $S$ to be the compositum $L'K$. Since $-1567 \equiv 1$ modulo $8$, $2$ splits completely in $K$. Then, for the number field $S/\Q$, we have that $f_2=7$, where $f_2$ is the inertia degree of $2$. Let us recall the function (\ref{f}) again.
\begin{displaymath}
f=2\sum_{\p}\sum^{\infty}_{m=1}\frac{\log N(\p)}{N(\p)^{m/2}}F(\log N(\p)^m).
\end{displaymath}
Since every term of $f$ is greater than or equal to $0$, the following holds for the number field $S$.
\begin{equation}
\begin{split}
f \geq 2\sum^{9216}_{j=1}\sum^{100}_{i=1}\frac{\log N(\bar{\q}_j)}{N(\bar{\q}_j)^{i/2}}F(\log N(\bar{\q}_j)^i),
\end{split}
\end{equation}
where the $\bar{\q}_j$ denote the prime ideals of $S$ satisfying $\bar{\q}_j | 2$. Since $f_2=7$, $N(\bar{\q}_j)=2^7$ for all $j$. Set $b=11.6$. By a numerical calculation, we have
\begin{equation}
\begin{split}
f \geq 2\cdot 9216 \sum^{100}_{i=1}\frac{\log 2^7}{2^{7i/2}}F(\log 2^{7i})=6814.41....
\end{split}
\end{equation}
Let us recall (\ref{D}). For $b=11.6$, we have
\begin{equation}
\begin{split}
|d_{S}|^{1/n_{S}}&>39.619 \cdot e^{(f-4790.3)/64512}\\
&\geq 39.619 \cdot e^{(6814.41-4790.3)/64512}=40.8818....
\end{split}
\end{equation}
Since $S/K$ is unramified, $|d_{S}|^{1/n_{S}}=|d_K|^{1/n_K}=\sqrt{1567}=39.5853....$. This is a contradiction.
Therefore, the $2$-class group of $M$ is trivial. In conclusion, the class number of $M$ is one.

\subsubsection{$A_5$-unramified extension of $M$}
Since $[K_{ur}:M]<66.1375..$, our final task is to show that $M$ does not admit an unramified $A_5$-extension.
By an analogous argument as in section \ref{A5}, $K$ admits an $A_5$-extension $F$ and $\Gal(F/\Q)$ is also isomorphic to $A_5 \times C_2$ or $S_5$.
\paragraph{Case 1 - $\Gal(F/\Q) \simeq A_5 \times C_2$ }
This implies that there exists an $A_5$-extension $F_1/\Q$ with ramification index $2$ at $1567$, and unramified at all other finite primes. However, from the tables in \cite{Basmaji-1994} no such extensions exists. This is a contradiction.
\paragraph{Case 2 - $\Gal(F/\Q) \simeq S_5$}
By the unramifiedness of $F/K$, a quintic subfield $E$ of $F$ must have the discriminant $-1567$. However the minimal negative discriminant of quintic fields with Galois group $S_5$ is $-4511$ (\cite[Table 3]{Schwarz-1994}. This is a contradiction.


Therefore, we know that $K_{ur}=M$ under the assumption of the GRH. This concludes the proof of Theorem \ref{thm:1567}.
\\
\\
{\textbf{Acknowledgments.}} The second author was supported by the Israel Science Foundation (grant No. 577/15).
\bibliographystyle{amsplain}

\end{document}